\documentclass[12pt]{article}
\usepackage[a4paper, margin=0.98in]{geometry}
\usepackage{amsmath,amssymb,amsthm,graphicx,fixmath,latexsym,color}
\usepackage{hyperref}

\definecolor{nicegreen}{rgb}{0,0.5,0}
\definecolor{green}{rgb}{0,0.5,0.5}

\def\R{{\mathbb R}}
\def\Re{{\mathbb R}}

\def\eps{{\varepsilon}}

\newtheorem{theorem}{Theorem}[section]
\newtheorem{lemma}[theorem]{Lemma}

\newtheorem{corollary}[theorem]{Corollary}

\makeatletter\let\@fnsymbol\@arabic\makeatother

\begin{document}
\title{Distinct distances between points and lines\let\thefootnote\relax\footnotetext{
Part of this research was performed while the first, second, and fourth authors were visiting the Institute for Pure and Applied Mathematics (IPAM) in Los Angeles, which is supported by the National Science Foundation.}
}

\author{Micha Sharir\thanks{School of Computer Science, Tel Aviv University, Tel Aviv 69978, Israel. 
Supported by BSF Grant 2012/229, by ISF Grant 892/13, by the I-CORE program (Center No.~4/11), and by the Hermann Minkowski--MINERVA Center for Geometry.}
\and
Shakhar Smorodinsky\thanks{Department of Mathematics, Ben-Gurion University of the NEGEV, Be'er Sheva 84105, Israel; and EPFL, Switzerland. 
Partially supported by Grant 1136/12 from the Israel Science Foundation, and by Swiss National Science Foundation Grants 200020144531 and 200021-137574.}
\and
Claudiu Valculescu\thanks{EPFL, Switzerland. Partially supported by Swiss National Science Foundation Grants 200020-144531 and 200021-137574.}
\and
Frank de Zeeuw\footnotemark[3]
}

\date{}
\maketitle

\vspace{-25pt}
\begin{abstract}
We show that for $m$ points and $n$ lines in $\R^2$, the number of distinct distances between the points
and the lines is $\upOmega(m^{1/5}n^{3/5})$, as long as $m^{1/2}\le n\le m^2$.
We also prove that for any $m$ points in the plane, not all on a line, the number of
distances between these points and the lines that they span is $\upOmega(m^{4/3})$.
The problem of bounding the number of distinct point-line distances can be reduced to the problem of bounding the number of tangent pairs among a finite set of lines and
a finite set of circles in the plane,
and we believe that this latter question is of independent interest.
In the same vein, we show that $n$ circles in the plane determine at most $O(n^{3/2})$ points where
two or more circles are tangent, improving the previously best known bound of $O(n^{3/2}\log n)$.
Finally, we study three-dimensional versions of the distinct point-line distances problem, namely, distinct point-line distances and distinct point-plane distances.
The problems studied in this paper are all new, and the bounds that we derive for them, albeit
most likely not tight, are non-trivial to prove. 
We hope that our work will motivate further studies of these and related problems.
\end{abstract}


\section{Introduction}\label{sec:intro}
In 1946 Paul Erd{\H o}s~\cite{Er46} posed the following two problems, which later became
extremely influential and central questions in combinatorial geometry:
the so-called \emph{repeated distances} and \emph{distinct distances} problems.
The first problem deals with the maximum number of repeated distances in a set of $n$ points in the plane, or in other words, the maximum number of
pairs of points at some fixed distance. 
The best known upper bound is
$O(n^{4/3})$~\cite{SST}, but the best known lower bound, attained by the
$\sqrt{n}\times\sqrt{n}$ grid, is only $\upOmega(n^{1+ \frac{c}{\log \log n}})$~\cite{Er46}.
The second problem asks for the minimum number of distinct distances determined by a set of $n$ points in the plane.
Guth and Katz~\cite{GK} proved that the number of distinct distances is
$\upOmega(n/\log n)$ for any set of $n$ points in the plane. This bound is nearly tight in the
worst case, since an upper bound $O(n/\sqrt{\log n})$ is attained by the $\sqrt{n}\times\sqrt{n}$ grid~\cite{Er46}.

In this paper we consider questions similar to those above, but for distances \emph{between points and lines}.
To be precise, we define the distance between a point $p$ and a line $\ell$ to be the minimum Euclidean distance between $p$ and a point of $\ell$.

Let $P$ be a set of $m$ points and $L$ a set of $n$ lines in the plane.
Let $I(P,L)$ denote the number of incidences between $P$ and $L$, i.e., the number of pairs $(p,\ell)\in P\times L$ such that the point $p$ lies on the line $\ell$.
The classical result of Szemer\'edi and Trotter \cite{ST83} asserts that
\begin{equation} \label{eq:st}
I(P,L) = O\left(m^{2/3}n^{2/3}+ m+n\right) .
\end{equation}
This bound is tight in the worst case, by constructions due to Erd{\H o}s and to Elekes.
See the survey of Pach and Sharir \cite{PS-incsurvey} for more details on geometric incidences.

The point-line incidence setup can be viewed as a special instance
of a repeated distance problem between points and lines.
Specifically, the Szemer{\' e}di-Trotter
result provides a sharp bound on the number of point-line pairs such that the point is at distance $0$ from the line.
As a matter of fact, the same bound holds if we consider pairs $(p,\ell)\in P\times L$
that have any fixed positive distance, say $1$. Indeed, replace each line $\ell\in L$
by a pair $\ell^+$, $\ell^-$ of lines parallel to $\ell$ and at distance $1$ from it.
Then any point $p\in P$ at distance $1$ from $\ell$ must lie on one of these lines.
Hence the number of point-line pairs at distance $1$ is at most the number of incidences
between the $m$ points of $P$ and the $2n$ lines $\ell^+$, $\ell^-$, for $\ell\in L$.
(Actually, a line in the shifted set might arise twice, but this does not affect the asymptotic bound.)
This proves that the number of times a single distance can occur between $m$ points and $n$ lines is
$O\left(m^{2/3}n^{2/3}+ m+n\right)$, and this bound is easily seen to be tight in the worst case.
Indeed, as mentioned above, there are sets of points and lines with this number of incidences, 
and by replacing each line with a parallel line at distance $1$, we get a construction with $\Theta\left(m^{2/3}n^{2/3}+ m+n\right)$ repeated point-line distances.

\paragraph{Distinct point-line distances.}
Our first main result concerns distinct distances between $m$ points and $n$ lines in the plane.
In contrast with the repeated distances question, as discussed above,
the distinct distances variant seems harder than for point-point distances,
and the lower bound that we are able to derive is inferior to that of \cite{GK}.
Nevertheless, deriving this bound is not an easy task, and follows by a combination
of several advanced tools from incidence geometry. We hope that our work will trigger
further research into this problem.

We write $D(m,n)$ for the minimum number of point-line distances determined by a set of $m$ points and a set of $n$ lines
in $\R^2$. Our first main theorem is the following lower bound for $D(m,n)$.

\begin{theorem}
\label{thm:main1}
For $m^{1/2}\le n\le m^2$, the minimum number $D(m,n)$ of point-line distances between $m$ points and $n$ lines in $\R^2$ satisfies
\[D(m,n) = \upOmega\left(m^{1/5}n^{3/5}\right).\]
\end{theorem}

Our proof also yields a stronger statement: For any set $P$ of $m$ points,
and any set $L$ of $n$ lines in the plane,
with $m^{1/2}\le n\le m^2$,
there always exists a point $p \in P$ such
that the number of distinct distances from $p$ to $L$ satisfies the bound in Theorem \ref{thm:main1}.

We note that the upper bound $D(m,n)\leq \lceil n/2 \rceil$ is easy to achieve by the following construction.
Place $n$ parallel lines, say the horizontal lines $y = j$ for integers $j=1,\ldots,n$.
Place all points on the line $y = n/2+1/2$. Since all points on the median line have the same distance
from any given horizontal line, the number of distinct point-lines distances is $\lceil n/2 \rceil$.
Note that this upper bound does indeed dominate the bound in Theorem~\ref{thm:main1}, as long as $m^{1/2}\le n$.

As a first step towards Theorem \ref{thm:main1}, we study the problem of bounding from above the number of tangencies between $n$ lines and $k$ circles; see Theorem \ref{thm:main3} below.
This results in a bound that is somewhat weaker than that in Theorem \ref{thm:main1}.
To arrive at the stronger bound in Theorem \ref{thm:main1}, we adapt an idea used by Sz\'ekely \cite{S} to prove the lower bound $\upOmega(n^{4/5})$ for the number of distinct point-point distances determined by $n$ points.

The bound of Sz\'ekely was improved by Solymosi and T\'oth \cite{ST} and by Guth and Katz \cite{GK},
but it appears difficult to adapt their proofs to the point-line distances setup.
Let us explain the problem in the case of the Guth-Katz proof.
A key step in that proof is a transformation from a pair of points $p,q$ to the set of rigid motions that map $p$ to $q$.
In the right parametrization, this set is a line in $\R^3$, and distinct pairs of points correspond to distinct lines.
To do the same for point-line distances, we would transform a pair of lines $\ell,\ell'$ to the rigid motions that map $\ell$ to $\ell'$,
which is again a line in $\R^3$.
However, it is not true that distinct pairs of lines always yield distinct lines in $\R^3$, and this ruins the Guth-Katz approach.
It thus appears that the point-line distances problem is in some sense harder, and obtaining a better bound seems to require significant new ideas.


\paragraph{Distinct distances between points and their spanned lines.}
We also study the number of distinct point-line distances between a finite set of non-collinear
points (that is, not all points lie on a common line) and the set of lines that they span.
This question has a different flavor, because the number of lines spanned by $m$ non-collinear points varies from $m$ to $\binom{m}{2}$.
When the points span many lines,
Theorem \ref{thm:main1} provides a reasonable bound,
which can be as large as $\upOmega(m^{7/5})$,
but when the points span few lines,
the resulting bound is a relatively weak $\upOmega(m^{4/5})$.
We use a different approach to obtain a better overall bound.

We write $H(m)$ for the minimum number of distances between $m$ points in $\R^2$, not all collinear, and the lines spanned by these points.
Note that $H(m)$ is also the minimum number of heights occurring in the triangles determined by a set of $m$ non-collinear points.

\begin{theorem}\label{thm:main2}
The minimum number $H(m)$ of distances between $m$ non-collinear points and their spanned lines satisfies
$$H(m) = \upOmega\left(m^{4/3}\right).$$
\end{theorem}

An upper bound $H(m)\leq m^2$ follows from a simple construction.
Place $m-1$ points on a line, and one point off the line.
This configuration spans only $m$ lines, and therefore it determines at most $m^2$ point-line distances.
The same bound also holds for other constructions, like the vertex set of a regular polygon, or an integer grid.

The lower bound in Theorem \ref{thm:main2} is most likely not tight, but it currently seems hard
to improve, for the following reason.
In the extreme configuration with $m-1$ points on a line, say the $x$-axis, and one point off the line,
say at $(0,1)$, it corresponds to a lower bound on the number of distinct values of the rational function
$f(x,y) = (x-y)^2/(1+y^2)$, with $x,y$ from a set $S\subset \R$ of size $m$.
Even for simpler functions, such as bivariate polynomials in $x,y$,
no better bound than $\upOmega(m^{4/3})$ is known (see, e.g., \cite{RSS}).

\paragraph{Tangencies involving lines and circles.}
As mentioned, our proof of Theorem \ref{thm:main1} is based on an analysis of the number of tangencies between lines and circles.
We believe this question to be of independent interest, and we consider it in more detail.

Given a finite set $L$ of lines and a finite set $C$ of circles, we write $T(L,C)$ for the number of \emph{tangencies} between lines from $L$ and circles from $C$, i.e., the number of pairs $(\ell,c)\in L\times C$ such that the line $\ell$ is tangent to the circle $c$.
We prove the following upper bound.

\begin{theorem} \label{thm:main3}
Let $L$ be a set of $n$ lines and $C$ a set of $k$ circles in the plane. Then
$$
T(L,C)=
O\left( n^{2/3}k^{2/3} + n^{6/11}k^{9/11}\log^{2/11} k + k + n \right) .
$$
\end{theorem}

We can obtain a lower bound for the maximum number of tangencies using a construction similar to the one that we gave for repeated point-line distances.
Given $m$ points and $n$ lines with $\Theta\left(m^{2/3}n^{2/3}+ m+n\right)$ incidences, replace each point by a circle of radius $1$, and replace each line by a parallel line at distance $1$.
The resulting $n$ lines and $k=m$ circles have $\Theta\left(n^{2/3}k^{2/3}+ k+n\right)$ tangencies.

From Theorem \ref{thm:main3} we can deduce a lower bound on $D(m,n)$, the minimum number of point-line distances, but the resulting bound
(see Corollary \ref{cor:lensbound} below) is weaker than that in Theorem \ref{thm:main1}.
In the proof of Theorem \ref{thm:main1} we improve this bound by exploiting the specific structure of the distinct distances problem.

It is natural to consider the corresponding question for tangencies between circles:
Given $n$ circles in the plane, how many pairs of circles can be tangent?
This is related to the problem of bounding the number of pairwise non-overlapping
\emph{lenses} in an arrangement of $n$ circles in the plane, which is of central significance
in the derivation of the best known bound on the number of point-circle incidences; see \cite{ANPPSS}.
The best known upper bound for both lenses and tangencies is $O(n^{3/2}\log n)$ (see \cite{MT}),
and the best known lower bound for both problems is $\upOmega(n^{4/3})$.
Under appropriate restrictions, we can slightly improve the upper bound to $O(n^{3/2})$ for the circle tangency problem.

\begin{theorem}\label{thm:circles1}
Let $C$ be a family of $n$ circles in $\R^2$ with arbitrary radii.
Assume that no three circles of $C$ are mutually tangent at a common point.
Then $C$ has at most $O(n^{3/2})$ pairs of tangent circles.
\end{theorem}

We also consider the following variant of this question, which bounds the number of \emph{tangency points} without any conditions on the circles.

\begin{theorem}\label{thm:circles2}
Given a family $C$ of $n$ circles in $\R^2$ with arbitrary radii,
the number of points where at least two circles of $C$ are tangent is $O(n^{3/2})$.
\end{theorem}

The same two results are proved in an independent concurrent work by Solymosi and Zahl \cite{SZ}.
Their proof is very different, and it works in greater generality for complex algebraic curves of bounded degree.
Their approach could also be used to obtain the bound $O(n^{3/2})$ on the number of tangencies between $n$ lines and $n$ circles, but this is weaker than our Theorem \ref{thm:main3}.

\paragraph{Distinct point-plane and point-line distances in $\R^3$.}
Finally, we consider two different three-dimensional variants of the question about distinct point-line distances.
Our main purpose is to introduce these problems, display some of the interesting geometry that they involve, and establish the first non-trivial bounds.
We start with \emph{point-plane} distances.

We first note that, without further assumptions, there is no non-trivial lower bound
on the number of distinct distances determined by $m$ points and $n$ planes.
Indeed, let $\Pi$ be a set of $n$ planes all containing a line $\ell$,
and let $P$ be a set of $m$ points on $\ell$.
Every point of $P$ is at distance $0$ from any plane in $\Pi$, so this configuration determines only one distance.
The construction can be modified to yield a set $P$ of $m$ points and a set $\Pi$ of $n$ planes, with only
one \emph{positive} point-plane distance, by placing all the points on the axis of a cylinder $C$,
and by choosing all the planes tangent to $C$.
Note also that the same example demonstrates that
there is no non-trivial upper bound on the number of repeated point-plane distances (without further assumptions).

We call a configuration of $s$ points and $t$ planes an $s\times t$  \emph{cone configuration} (resp., an $s\times t$ \emph{cylinder configuration})
if the $s$ points lie on a line $\ell$, and the $t$ planes are all tangent to the same cone (resp., cylinder) with axis $\ell$.
The next theorem shows that if we exclude such configurations, then there is a non-trivial lower bound on the number of distinct point-plane distances.
Note that by the example above, cylinder configurations have to be excluded to get a non-trivial bound.
However, this does not seem to be the case for cone configurations; these show up as an artefact of our proof (see Section \ref{sec:pointplane} for details).

\begin{theorem}
\label{thm:3Dplanes}
Let $P$ be a set of $m$ points and let $\Pi$ be a set of $n$ planes in $\Re^3$.
Assume that $P$ and $\Pi$ do not determine a $3\times 3$ cone or cylinder configuration.
Then the number of distinct point-plane distances determined by $P$ and $\Pi$ is
\[\upOmega\left(\min\left\{m^{1/3}n^{1/3},n \right\}\right),\]
unless $m=O(1)$.
\end{theorem}

The best upper bound for the minimum number of point-plane distances that we have is $n$, from the following construction.
Let $\Pi$ be a set of $n$ parallel planes, and let $P$ be a set of $m$ arbitrary points on one of the planes in $\Pi$.
Then each point of $P$ determines the same set of at most $n$ distances to the planes of $\Pi$.
Clearly, a large gap remains between this upper bound and the lower bound in Theorem \ref{thm:3Dplanes} (when $m<n^2$).

Next we consider distinct \emph{point-line} distances in $\R^3$.
Let us say that $s$ points and $t$ lines form an $s\times t$ \emph{cone/cylinder/hyperboloid configuration} if the $s$ points lie on a line $\ell$, and the $t$ lines are all contained in the same cone/cylinder/one-sheeted hyperboloid with axis $\ell$.
Note that these quadrics are \emph{ruled surfaces}.
Again a cylinder configuration determines only one point-line distance,
while the other configurations appear to be artefacts of our proof.

\begin{theorem}
\label{thm:3Dlines}
Let $P$ be a set of $m$ points and $L$ a set of $n$ lines in $\Re^3$.
Assume that $P$ and $L$ do not determine any $4 \times 4$ cone/cylinder/hyperboloid configuration.
Then the number of distinct point-line distances determined by $P$ and $L$ is
\[\upOmega\left(\min\left\{m^{1/4}n^{1/4-\eps},n\right\}\right),\]
for any $\eps>0$ (where the constant of proportionality depends on $\eps$), unless $m=O(1)$.
\end{theorem}

Taking $n$ parallel lines and $m$ points on one of these lines, we get a configuration with at most $n$ distinct point-line distances.

Finally, let us mention as an open problem one further variant in $\R^3$.
One can ask for a lower bound on the minimum number of \emph{line-line distances} determined by a finite set of lines in $\R^3$.
If the lines all lie in a plane or a cone, then they may all intersect, and thus determine only one line-line distance.
We believe that if no three lines are in a plane or cone, then a non-trivial lower bound should hold.
We could also consider a bipartite version of this problem, where we have two sets of lines and we consider only distances between a line from one set and a line from the other.
Here there is another degenerate configuration, if the two sets of lines are subsets of the two rulings of a doubly-ruled surface (see \cite{GK} for definitions).
In this case there is only one bipartite line-line distance, because every line from one ruling intersects every line from the other ruling (even though the lines within a single ruling do not intersect, and in fact determine many line-line distances).

\paragraph{Outline.}
In Section \ref{sec:prelims} we introduce the incidence bounds that are key tools in our proofs, and we state a number of geometric facts that are needed to apply the incidence bounds.
In Section \ref{sec:pointline}, we prove Theorem \ref{thm:main1} on point-line distances, and along the way we prove Theorem \ref{thm:main3} on line-circle tangencies.
Section \ref{sec:spanned} then concerns Theorem \ref{thm:main2} on distances between points and spanned lines, and Section \ref{sec:circletangs} treats tangencies between circles.
Finally, Section \ref{sec:highdim} considers point-plane and point-line distances in $\R^3$.

\paragraph{Acknowledgement.}
We would like to thank Joshua Zahl for helpful discussions concerning the problems studied in this paper, and the referees of an earlier version of this paper for their helpful comments.


\section{Preliminaries}
\label{sec:prelims}

\subsection{Incidence bounds}
\label{subsec:incs}

Given a finite set $P$ of points and a finite set $\mathcal{O}$ of geometric objects, the \emph{incidence graph} of
$P\times \mathcal{O}$ is the bipartite graph with vertex sets $P$ and $\mathcal{O}$, and an edge between $p\in P$ and $O\in\mathcal{O}$ if $p\in O$.
The number of incidences between $P$ and $\mathcal{O}$, denoted by $I(P,\mathcal{O})$, is the number of edges in this graph.
The following incidence bound of Pach and Sharir \cite{PS98} generalizes the theorem of Szemer\'edi and Trotter \cite{ST83} mentioned in the introduction.
The version stated here, which is a special case of the more general result of \cite{PS98}, can be found in \cite{KMS}.

\begin{theorem}[Pach-Sharir]\label{thm:pachsharir}
Let $P\subset\R^2$ be a set of points and let $\mathcal{C}$ be a set of algebraic curves in $\R^2$ whose degree is bounded by a constant.
Assume that the incidence graph of $P\times \mathcal{C}$ contains no $K_{s,M}$, for some constant $M$.
Then
\[I(P,\mathcal{C}) = O\left(|P|^{\frac{s}{2s-1}}|\mathcal{C}|^{\frac{2s-2}{2s-1}} + |P|+|\mathcal{C}| \right). \]
\end{theorem}

The following theorem of Agarwal et al. \cite{ANPPSS} gives an improvement on the Pach-Sharir bound for a specific class of curves.
In the language of \cite{ANPPSS}, a family of \emph{pseudo-parabolas} is a family of graphs of everywhere defined
continuous univariate functions, so that each pair intersect in at most two points.
A collection of pseudo-parabolas \emph{admits a three-parameter representation} if the curves have three degrees of freedom, and can thus be identified with points in $\R^3$ in a suitable manner.
A full definition is given at the beginning of Section 5 in \cite{ANPPSS}.

\begin{theorem}[Agarwal et al.]\label{thm:agarwal}
Let $P\subset\R^2$ be a set of points and let $\mathcal{C}$ be a set of
pseudo-parabolas that admit a three-parameter representation.
Assume that the incidence graph of $P\times \mathcal{C}$ contains no $K_{s,M}$, for some constant $M$.
Then
\[|I(P,\mathcal{C})| = O\left(|P|^{2/3}|\mathcal{C}|^{2/3} + |P|^{6/11}|\mathcal{C}|^{9/11}\log^{2/11}|\mathcal{C}| + |P|+|\mathcal{C}| \right). \]
\end{theorem}

We will also use the crossing lemma (see Sz\'ekely \cite{S}) and a theorem of Beck \cite[Theorem 3.1]{B} (sometimes referred to as ``Beck's two extremities theorem'').

\begin{theorem}\label{thm:crossing}
A graph $G$ with $n$ vertices, $e$
edges, and maximum edge multiplicity $s$, has at least $\upOmega(e^3/sn^2)$ edge crossings
in any plane drawing, unless $e< 4ns$.
\end{theorem}

\begin{theorem}[Beck]\label{thm:beck}
There is a constant $c>0$ such that, for any set of $m$ points in $\R^2$, either at least $cm$
of the points lie on a line, or the points determine at least $cm^2$ distinct lines.
\end{theorem}

In three dimensions, we use the following incidence bound due to Zahl \cite{Z}
(the same bound was independently proved in the case of unit spheres by Kaplan et al. \cite{KMSS}).
We state it in a slightly different form that is convenient for us, and that follows from the proof given in \cite{Z}.
We write $E(G)$ for the set of edges of a graph $G$.

\begin{theorem}[Zahl]\label{thm:zahl}
Let $P\subset\R^3$ be a set of $m$ points and let $\mathcal{S}$ be a set of $n$ algebraic
surfaces in $\R^3$ whose degree is bounded by a constant.
Let $I$ be a subgraph of the incidence graph of $P\times \mathcal{S}$ that contains neither $K_{M,3}$ nor $K_{3,M}$,
for some constant parameter $M$.  Then
\[|E(I)| = O\left(m^{3/4}n^{3/4}+m+n\right),\]
where the constant of proportionality depends on the maximum  degree of the surfaces and on $M$.
\end{theorem}

Although our main theorems are set in at most three dimensions, some of our proofs involve incidence problems in higher dimensions.
We will use the following incidence bound from Fox et al.~\cite{FPSSZ}.
Note that Theorem \ref{thm:zahl} is essentially the special case where $d = 3$ and $s = 3$, without the extra $\eps$ in the exponent.
The case $d=4$ was independently proved by Basu and Sombra \cite{BS}, without the $\eps$, but under slightly more restrictive conditions.

\begin{theorem}[Fox et al.]\label{thm:FPSSZ1}
Let $\mathcal{V}$ be a set of $n$ constant-degree varieties and let $\mathcal{P}$ be a set of $m$ points, both in $\R^d$, such that the incidence graph of $\mathcal{P}\times\mathcal{V}$ contains no $K_{s,t}$.
Then
\[I(\mathcal{P},\mathcal{V}) = O\left(m^{\frac{s(d-1)}{ds-1}+\eps}n^{\frac{d(s-1)}{ds-1}} +m+n\right),\]
for any $\eps>0$, where the constant of proportionality depends on $d,s,t, \eps$.
\end{theorem}

\subsection{Geometric facts}\label{sec:geometricfacts}

We need a number of basic facts about tangencies between various geometric objects.
For some of the more elementary facts we provide quick (and not too rigorous) proofs, while for the more difficult facts we only give references.

We define two algebraic varieties to be \emph{tangent} to each other at a point $p$ if both are nonsingular at $p$, and the tangent space at $p$ of one of the varieties is contained in the tangent space at $p$ of the other.
We will use this notion only for simple varieties, namely lines and circles in $\R^2$, and lines, planes, and quadrics in $\R^3$.

\begin{lemma}\label{lem:twocircles}
$(a)$ Any two circles in $\R^2$ have at most four common tangent lines.\\
$(b)$ Any three lines in $\R^2$ have at most four common tangent circles.
\end{lemma}
\begin{proof}
Part $(a)$ is elementary.
For part $(b)$, observe that if the three lines are concurrent, then there is no tangent circle.
Otherwise, we can move to the projective plane and observe that the lines divide the plane into four triangular cells.
Since a circle that is tangent to a line cannot cross the line, a common tangent circle must lie inside such a cell. Since a triangle contains exactly one circle tangent to all its sides, there are at most four common tangent circles.
\end{proof}

The ``problem of Apollonius'' asks for a circle tangent to three given circles.
All solution sets to this problem have been classified; see for instance \cite{BFW}. The lemma below follows from this classification.

\begin{lemma}\label{lem:apollonius}
Given three circles in $\R^2$ that are not mutually tangent at a common point, there are at most eight other circles tangent to all three.
\end{lemma}

We will need several three-dimensional versions of Lemma \ref{lem:twocircles}, both for spheres and planes and for spheres and lines.
The statements become more complicated, as we encounter configurations for which no analogous bounds hold.
Such configurations are related to those in Theorems \ref{thm:3Dplanes} and \ref{thm:3Dlines}.
See Section \ref{sec:intro} for the definitions of the configurations.

\begin{lemma}\label{lem:threespheres}
$(a)$ Three spheres in $\R^3$ have at most eight common tangent planes, unless the spheres have a common tangent cone or cylinder.\\
$(b)$ Given three planes and at least nine common tangent spheres,
there are three spheres out of the nine that together with the three planes form a  $3\times 3$ cylinder or cone configuration.
\end{lemma}
\begin{proof}
In projective space, any two spheres have at most  two common tangent cones, and all common tangent planes of the two spheres are tangent to one of those cones.
We claim that one such cone and a third sphere have at most four common tangent planes, unless the sphere is tangent to the cone (which may be a cylinder in $\R^3$).
This claim implies part $(a)$.
To prove the claim, apply a projective transformation that maps the cone to a cylinder and the sphere to a sphere. 
Intersect both with the plane through the center of the sphere and orthogonal to the axis of the cylinder, so that we get two circles, unless the sphere is tangent to the cylinder.
A common tangent plane of the cylinder and the sphere corresponds to a common tangent line of the two circles.
By Lemma \ref{lem:twocircles}$(a)$, there are at most four such lines.
This proves the claim. 

Consider part $(b)$.
If the intersection of the three planes is a line, then there is no common tangent sphere.
Otherwise, in projective space, the three planes intersect in a single point, and after a projective transformation we can assume that this point lies in $\R^3$.
The planes divide $\R^3$ into eight unbounded cells, each with three faces.
The set of points equidistant from all three planes consists of four lines through the origin, and each of these lines is the axis of a cone that is tangent to all three planes.
Any common tangent sphere must lie inside one of the cells, and it must be tangent to one of the four cones.
Thus, if there are more than eight common tangent spheres, there are three that are tangent to the same cone.
Together with the three planes these three spheres form a $3\times 3$ cone configuration.
After inverting the projective transformation, this may be a cylinder configuration in the original space.
\end{proof}

Finally, we need an analogue of Lemma \ref{lem:twocircles} for spheres and lines,
which is considerably harder, and was proved by Borcea et al. \cite{BGLP06}.
Note that their Theorem 1 only states that four spheres with collinear centers are exceptional, but their Lemma 7 classifies the exceptional configurations as spheres with a common cone, cylinder, or one-sheeted hyperboloid.
Note that a cone can be seen as a degenerate one-sheeted hyperboloid, and a cylinder can be seen as a cone with apex at infinity.
Each of these surfaces contains infinitely many lines (they are ruled surfaces), and all these lines are tangent to any sphere that is tangent to the surface and has its center on the axis of the surface.

\begin{lemma}[Borcea et al.]\label{lem:fourspheres}
Four spheres in $\R^3$ have at most twelve common tangent lines, unless the spheres have collinear centers.
If the spheres have collinear centers, then any common tangent line lies on a cone, cylinder, or one-sheeted hyperboloid tangent to all the spheres.
\end{lemma}


\section{Distinct distances between points and lines}
\label{sec:pointline}

In this section we provide a lower bound on the minimum number $D(m,n)$ of distinct point-line
distances determined by $m$ points and $n$ lines in the plane.
As warmup, we first give a weaker bound in Subsection \ref{subsec:firstbound}, using a proof that serves as a model for most of the later proofs.
In Subsection \ref{subsec:linecirctangs}, we prove an upper bound on the number of tangencies
between a set of lines and a set of circles,
and as a corollary we obtain a slightly better
lower bound on $D(m,n)$.
For both of these bounds, we translate the problem to one about tangencies between circles and lines,
and then we dualize the problem to an incidence problem between points and algebraic curves, to which we can apply a known bound.
In Subsection \ref{subsec:mainbound}, we use the special structure of the problem to derive a better bound on $D(m,n)$, which is the bound in Theorem \ref{thm:main1}.


\subsection{An initial bound}\label{subsec:firstbound}

Let $P$ be a set of $m$ points and $L$ a set of $n$ lines.
Let $t$ be the total number of distinct distances between $P$ and $L$.
Around each of the $m$ points,
draw at most $t$ circles whose radii are the distances occurring from that point, and let $C$ be the set of these circles.
Recall that we write $T(L,C)$ for the number of tangencies, i.e., pairs $(l,c)\in L\times C$ such that the line $\ell$ is tangent to the circle $c$.
Note that $T(L,C) = mn$ since every point-line pair gives rise to one tangency between a line of $L$ and a circle of $C$.

We now dualize.
Concretely, we rotate the plane so that none of the lines in $L$ is vertical, and then we map a line $y=ax+b$ to the dual point $(a,b)$.
Under this map, an algebraic curve is mapped to the set of points that are dual to the non-vertical tangent lines of the curve; these dual points form an algebraic curve, called the \emph{dual curve}.
We refer to the original $xy$-plane as the \emph{primal plane}, and to the $ab$-plane as the \emph{dual plane}.

Applying this to our setting, the set $L$ of $n$ lines in the primal plane is mapped to a set $L^*$ of $n$ points in the dual plane,
and the set $C$ of (at most) $mt$ circles in the primal plane is mapped to a set $C^*$ of (at most) $mt$ algebraic curves in the dual plane.

We observe that by Lemma \ref{lem:twocircles}(b) the incidence graph of $L^*\times C^*$ contains no $K_{3,5}$.
Given this property, we can apply Theorem \ref{thm:pachsharir} with $s =3$ to get
\[I(L^*, C^*) = O(n^{3/5}(mt)^{4/5}+n+mt). \]
On the other hand, we have $I(L^*, C^*)=T(L,C) = mn$. Comparing these bounds gives either $mn = O(m^{4/5}n^{3/5}t^{4/5})$, so $t = \upOmega(m^{1/4}n^{1/2})$;
or $mn =O(n)$, so $m=O(1)$;
or $mn = O(mt)$, so $t = \upOmega(n)$.
Thus
\[D(m,n) = \upOmega\left(\min\left\{m^{1/4}n^{1/2},n \right\}\right),\]
unless $m=O(1)$.
We do not even state this bound as a theorem, because it will be superseded by the bounds that follow.
But we note that our three-dimensional proofs in Section \ref{sec:highdim} are modelled after this argument.


\subsection{A bound on line-circle tangencies}\label{subsec:linecirctangs}

We now give the proof of Theorem \ref{thm:main3}, which we restate here for the convenience of the reader.


\begingroup
\def\thetheorem{\ref{thm:main3}}
\begin{theorem}
Let $L$ be a set of $n$ lines and $C$ a set of $k$ circles in the plane. Then
\[
T(L,C)=
O\left(n^{2/3}k^{2/3} + n^{6/11}k^{9/11}\log^{2/11} k + k + n \right).
\]
\end{theorem}
\addtocounter{theorem}{-1}
\endgroup
\begin{proof}
We dualize as in Subsection \ref{subsec:firstbound}.
Thus the set $L$ of $n$ lines in the primal plane is mapped to a set $L^*$ of $n$ points in the dual plane,
and the set $C$ of $k$ circles in the primal plane is mapped to a set of $k$ algebraic curves in the dual plane.
We claim that these curves are hyperbolas.

Indeed, the dual curve $c^*$ of a circle $c$ is the locus of all points $(a,b)$ dual to lines that are tangent to $c$.
If $c$ is centered at a point $p=(p_1,p_2)$ and has radius $r$, then the equation in $a,b$ that defines $c^*$ is (using the standard formula for the distance between a point and a line)
$$
\frac{|p_2-p_1a-b|}{\sqrt{1+a^2}} = r ,\quad\quad\text{or}\quad\quad
(p_2-p_1a-b)^2 - r^2(1+a^2) = 0 ,
$$
which is the equation of a hyperbola.
We only need this equation to establish that the dual curve is a hyperbola, and we do not use it any further; instead, we derive properties of these hyperbolas from the corresponding setting in the primal plane.

We treat each branch of $c^*$ as a separate curve, and obtain a
collection $C^*$ of $2k$ such curves.
Standard considerations show that one of the two branches of $c^*$ is the locus of points
dual to the lines tangent to $c$ from above, and the other branch
is the locus of points dual to the lines tangent to $c$ from below.

Each tangency between a line $\ell$ and a circle $c$ corresponds to an incidence between the dual hyperbola $c^*$ and the point $\ell^*$.
It is easily checked that the curves in $C^*$ are \emph{pseudo-parabolas}, since any two hyperbola branches in $C^*$ intersect
at most twice (indeed, any pair of circles have at most two common tangents, if for each circle we consider only the tangents on one side, i.e., above or below). 
Moreover, the
curves in $C^*$ admit a three-parameter representation (see Subsection \ref{subsec:incs}), because the
circles in $C$ have such a representation. Then the desired bound follows from Theorem \ref{thm:agarwal}.
\end{proof}


We now deduce a lower bound on distinct point-line distances from the line-circle tangency bound above.
It is better than that in Subsection \ref{subsec:firstbound}, but in Subsection \ref{subsec:mainbound}, we will significantly improve it.

\begin{corollary} \label{cor:lensbound}
The minimum number $D(m,n)$ of point-line distances between $m$ points and $n$ lines in $\R^2$ satisfies
$$
D(m,n) = \upOmega\left(m^{2/9}n^{5/9}\log^{-2/9} m\right),
$$
provided that $m^{1/2}/\log^{1/2}m \le n \le m^5\log^4m$, and that $m$ is larger than some sufficiently large constant.
\end{corollary}
\begin{proof}
Let $P$ be a set of $m$ points and let $L$ be a set of $n$ lines in the plane.
Let $t$ be the number of distinct distances between points of $P$ and lines of $L$.
For each point $p \in P$, draw at most $t$
circles centered at $p$ with radii equal to the at most $t$ distances from
$p$ to the lines in $L$.
We obtain a family $C$ of at most $mt$ distinct
circles.

We double-count $T(L,C)$.
On the one hand, we trivially have $T(L,C) = mn$,
since for each of the $mn$ point-line pairs $(p,\ell) \in P \times L$, there
is exactly one tangency between the line $\ell$ and the circle centered at $p$ whose radius is the distance from $p$ to $\ell$.
On the other hand, we can apply Theorem \ref{thm:main3} with $|C| \leq mt$ and $|L| = n$ to obtain
$$
mn = T(L,C) = O\left( n^{2/3}(mt)^{2/3} + n^{6/11}(mt)^{9/11}\log^{2/11}(mt) + mt + n \right) .
$$
Eliminating $t$ from this inequality yields 
\[
t = \upOmega\left(\min\left\{m^{1/2}n^{1/2},\; n^{5/9}m^{2/9}\log^{-2/9} m,\; n \right\} \right),
\]
assuming that $m$ is at least some sufficiently large constant.
The minimum is attained by the second term, unless either $n > m^5\log^4m$ or $n < m^{1/2}/\log^{1/2}m$.
\end{proof}


\subsection{Proof of Theorem \ref{thm:main1}}\label{subsec:mainbound}

In this section we prove Theorem \ref{thm:main1}, which we restate below.
We start, as in the proof of Corollary \ref{cor:lensbound}, by reducing the problem to counting line-circle tangencies, and then we dualize.
Instead of directly using an incidence bound for the dual curves, we derive a better bound by taking a closer look at the structure of the problem.
In particular, we make use of the fact that our family of circles consists of relatively few families of many  concentric circles.
Our approach is similar to that used by Sz\'ekely \cite{S} to prove the bound $\upOmega(m^{4/5})$ on the number of distinct point-point distances determined by $m$ points,
but our proof is more involved because we have to go back and forth between the primal and dual plane.

\begingroup
\def\thetheorem{\ref{thm:main1}}
\begin{theorem}
For $m^{1/2}\le n\le m^2$, the minimum number $D(m,n)$ of point-line distances between $m$ points and $n$ lines in $\R^2$ satisfies
\[D(m,n) = \upOmega\left(m^{1/5}n^{3/5}\right).\]
\end{theorem}
\addtocounter{theorem}{-1}
\endgroup
\begin{proof}
Let $P$ be a set of $m$ points and let $L$ be a set of $n$ lines in the plane.
Again, let $t$ denote the number of distinct point-line distances, draw at most $t$ circles around every point
of $P$, and denote the resulting set of circles by $C$.  As before, we have $T(L,C)= mn$.

In the dual plane we have a set $L^*$ of $n$ points.
Recall that the dual curve $c^*$ of a circle $c$ is a hyperbola.
As before, we treat each branch of the hyperbola as a separate curve, and we let $C^*$ be the set of these $2mt$ hyperbola branches.

To bound the number $I(L^*,C^*)$ of incidences between $L^*$ and $C^*$, we draw a topological (multi-)graph $G=(L^*,E)$ in the dual plane.
We assume without loss of generality that each hyperbola branch in $C^*$ contains at least two points of $L^*$.
Indeed, we can discard all curves of $C^*$ containing at most one point of $L^*$, thereby discarding at most $2mt$ incidences.

For every curve in $C^*$, we connect each pair of consecutive points of $L^*$ on that curve by an edge drawn along the portion of the curve between the two points.
Write $E$ for the set of edges in this graph.
The number of edges on each curve of $C^*$ is exactly one less than the number of points on it, so overall the number of edges in $G$ satisfies
$$|E| \geq I(L^*,C^*)-2mt.$$

Note that an edge can have high multiplicity, when many curves of $C^*$ pass through its two endpoints, and the endpoints are consecutive on each of these curves.
This situation corresponds to the case in the primal plane where we have many circles touching a pair of lines,
and moreover, the corresponding tangencies are consecutive on each of the circles.

Let $s$ be a parameter that will be chosen later.
Let $E_1$ denote the set of edges with multiplicity at most $s$ and let $E_2$ denote
the set of edges with multiplicity larger than $s$.  In order to bound $|E_1|$ we use the crossing lemma (Theorem \ref{thm:crossing}).
We apply it to the graph with vertex set $L^*$ and edge set $E_1$.
Since any pair of these hyperbola branches intersect at most twice,
the total number of crossings between curves in $C^*$ is at most $2\cdot \binom{2mt}{2} = O(m^2t^2)$.
Combining this upper bound with the lower bound from Theorem \ref{thm:crossing}, we get
\[m^2t^2 = \upOmega\left(\frac{|E_1|^3}{s n^2}\right),\]
so, taking into account the alternative case $|E_1|< 4 ns$,
\begin{equation}\label{eq:Eone}
|E_1| = O(m^{2/3}n^{2/3}t^{2/3}s^{1/3}+ns).
\end{equation}

Next, we consider the edges of $E_2$.  If an edge with endpoints $\ell_1^*,\ell_2^*$ has multiplicity $x$, then the lines $\ell_1$ and $\ell_2$ in the primal plane have $x$ common tangent circles.
The centers of these circles lie on the two angular bisectors defined by $\ell_1,\ell_2$, so
there must be at least $x/2$ incidences between the $m$ points and one of the bisectors of $\ell_1,\ell_2$.

We charge each edge of $E_2$ to the incidence between the angular bisector and the center of the circle
$c$ dual to the curve that the edge lies on.  We claim that each such incidence can be charged at most $2t$ times.
Indeed, in the primal plane, consider such an incidence between a point $p$ and and an angular bisector $\ell$.
There are at most $t$ distinct circles with the same center $p$, and each of these circles can have at most two
pairs of tangent lines such that the angular bisector of those lines is $\ell$, \emph{and} such that the tangencies are consecutive.
(In this argument, we acknowledge the possibility that $\ell$ might be the angular bisector of many pairs of lines,
all of which are tangent to the same circle.)

It follows from the Szemer{\' e}di-Trotter theorem (see \eqref{eq:st} or Theorem \ref{thm:pachsharir}) that the number of lines containing
at least $s/2$ of the $m$ points is $O(m^2/s^3+m/s)$ (as long as $s/2 > 1$), and that the number of incidences between
these $m$ points and $O(m^2/s^3+m/s)$ lines is $O(m^2/s^2+m)$. Thus, using the observation that each incidence is charged by at most $2t$ edges, we have
\begin{equation}\label{eq:Etwo}
|E_2| = O\left(\frac{m^2t}{s^2}+mt\right).
\end{equation}
To balance (\ref{eq:Eone}) and (\ref{eq:Etwo}), we choose $s = O(m^{4/7}t^{1/7}/n^{2/7})$,
noting that, with a proper choice of the constant of proportionality, we have $s>2$. Indeed,
this amounts to requiring that $n=O(m^2t^{1/2})$, which holds since $m\ge n^{1/2}$.
Adding together \eqref{eq:Eone} and \eqref{eq:Etwo} gives
\[|E| = O\left(m^{6/7}n^{4/7}t^{5/7} +m^{4/7}n^{5/7}t^{1/7}+ mt\right).\]
Thus the same bound holds for $T(L,C)$.
Combining this with $T(L,C) = mn$, we get $t = \upOmega(m^{1/5}n^{3/5})$ from the first term,
$t = \upOmega(m^3n^2)$ from the second term, and $t = \upOmega(n)$ from the third term. Thus
\[t = \upOmega\left(m^{1/5}n^{3/5}\right),\]
using the assumption $m\le n^2$.  This completes the proof of Theorem \ref{thm:main1}.
\end{proof}

Note that in the proof above we could set $t$ to be the maximum number of distances from one
of the $m$ points to the $n$ lines. We can therefore conclude that there is a single point from
which the number of distances satisfies the bound.
We note here that this is not the case in the proof of Theorem \ref{thm:main2}
that is provided in Section \ref{sec:spanned} below,
so that proof does not lead to a stronger conclusion of this type.


\section{Distances between points and spanned lines}\label{sec:spanned}

In this section we prove Theorem \ref{thm:main2}, restated below.
The problem is to bound from below the number of distinct distances between a point set $P$ and the set of lines spanned by $P$.
Equivalently, we want to bound from below the number of distinct heights of triangles spanned by $P$.
Write $\ell_{bc}$ for the line spanned by the points $b$ and $c$, write
\[H(P) = \left|\{d(a,\ell_{bc})\mid a,b,c\in P\}\right|\]
for the number of distances between points of $P$ and lines spanned by $P$, and write $H(m)$ for the minimum value of $H(P)$ over all non-collinear point sets $P$ of size $m$.

For point sets with not too many points on a line, Theorem \ref{thm:main1} gives a better bound than Theorem \ref{thm:main2} stated below.
Our goal here is to provide a reasonably good bound that holds also when many points are collinear.
To prove this, we reduce it to showing that the rational function
\[f(x,y)=\frac{(x-y)^2}{1+y^2}\]
is ``expanding'', in the sense that $f(x,y)$ takes $\upOmega(m^{4/3})$ distinct values for $x,y$ in any set of $m$ real numbers.
If $f$ were a polynomial, this would follow directly from a result of Raz et al. \cite{RSS}. However, for rational
functions in general the only known result is that of Elekes and R\'onyai \cite{ER}, which says that the number of values
is superlinear (both these statements have certain exceptions). To extend the bound $\upOmega(m^{4/3})$ to the rational
function $f$, we use the same approach as in \cite{RSS}, which originated in Sharir, Sheffer, and Solymosi \cite{SSS}.

\begingroup
\def\thetheorem{\ref{thm:main2}}
\begin{theorem}
The minimum number $H(m)$ of distances between $m$ non-collinear points and their spanned lines satisfies
\[ H(m) =\upOmega(m^{4/3}).\]
\end{theorem}
\addtocounter{theorem}{-1}
\endgroup
\begin{proof}
By Theorem \ref{thm:beck}, there is a constant $c$ such that either the points of $P$ span at least $cm^2$ distinct lines, or at least $cm$ points of $P$ are collinear.

In the first case, Theorem \ref{thm:main1} gives
\[H(P) \geq D(m, cm^2)
= \upOmega\left(m^{7/5}\right).\]
Consider the second case, when $k= cm$ of the points are collinear; we can assume that $k$ is an integer.
Since not all the points are collinear, at least one other point $q\in P$ does not belong to this line.
By translating, rotating, and scaling, we can assume that $q=(0,1)$ and that the $k$ collinear points
lie on the $x$-axis, and by removing at most half the points we can assume that they all lie on the
positive $x$-axis. We denote them by $p_i = (x_i,0)$ for $i=1,\ldots,k$, with all $x_i$ positive,
and we set $W=\{x_1,\ldots,x_k\}$.

By a standard formula, the distance $d(p_i,\ell_{p_jq})$ from a point $p_i$ to the line $\ell_{p_jq}$ spanned by $p_j$ and
$q$ equals $|x_i-x_j|/\sqrt{1+x_j^2}$. Thus, putting
\[f(x,y)=\frac{(x-y)^2}{1+y^2},\]
we get that $f(x_i,x_j) = d^2(p_i,\ell_{p_jq})$. Hence,
in order to obtain a lower bound for the number of point-line distances,
it suffices to find a lower bound for the
cardinality of the set $f(W) = \{ f(x,y)\mid x,y\in W\}$, for a set $W$ of $k=\upOmega(m)$
real positive numbers.

Following the setup in \cite{SSS}, we define the set of quadruples
\[ Q=\{(x,y,x',y') \in W^4\mid f(x,y)=f(x',y')\}.\]
Writing $f^{-1}(a)= \{(x,y)\in W^2 \mid f(x,y)=a \}$ and using the Cauchy-Schwarz inequality, we have
\[|Q| = \sum_{a\in f(W)} |f^{-1}(a)|^2 \ge \frac{k^4}{|f(W)|}.\]
Thus an upper bound on $|Q|$ will imply a lower bound on $|f(W)|$.

We define algebraic curves $C_{ij}$ in $\R^2$, for $i,j=1,\ldots,k$, by
\[C_{ij}=\{(z,z')\in \mathbb{R}^2\mid f(z,x_i)=f(z',x_j)\}.\]
We have $(x_k,x_l)\in C_{ij}$ if and only if $(x_k,x_i,x_l,x_j)\in Q$.
Thus, denoting by $\Gamma$ the set of curves $C_{ij}$, and by $S=W^2$ the set of pairs $(x_k,x_l)$, we have (in the notation of Subsection \ref{subsec:incs})
\[ |Q| = I(S,\Gamma).\]

It is not hard to show that the curves $C_{ij}$ with $i=j$ contribute at most $O(k^2)$ quadruples,
which is a negligible number, so in the rest of the proof we will assume that $i\neq j$.
The equation $f(z,x_i)=f(z',x_j)$ is equivalent to $(z-x_i)^2/(1+x_i^2)= (z'-x_j)^2/(1+x_j^2)$, or
\[ z' - x_j =\pm A_{ij}\cdot (z-x_i),\]
where
\[A_{ij} = \sqrt{(1+x_j^2)/(1+x_i^2)}.\]

Every curve $C_{ij}$ is thus the union of two lines in the $zz'$-plane, given by
\[ L_{ij}^{+} : \; z'= A_{ij}z +(x_j - A_{ij}x_i ),  \qquad
L_{ij}^{-} : \; z' =-A_{ij}z+(x_j + A_{ij}x_i ). \]
Therefore, we need only consider the two families $\Gamma^+ = \{L_{ij}^+\mid i\neq j\}$ and
$\Gamma^- = \{L_{ij}^-\mid i\neq j\}$ and bound $I(S,\Gamma^+\cup\Gamma^-)$.
We claim that the lines in $\Gamma^+$ and $\Gamma^-$ are all distinct.
Indeed, given the equation $z'=\alpha z+\beta$ of a line, say with
$\alpha>0$, we have $A_{ij}=\alpha$ and $x_j-\alpha x_i = \beta$. This leads,
as is easily verified, to the linear equation
$2\alpha\beta x_i = \alpha^2-\beta^2-1$ in $x_i$.
This equation has at most one solution in $x_i$ (and thus $x_j$ is also unique), unless
$2\alpha\beta = \alpha^2-\beta^2-1 = 0$. However, $\alpha\ne 0$ by its definition, so
the only problematic case is $\beta=0$, $\alpha=1$ (recall that we assume that $\alpha>0$).
This too is impossible, because we would then have $x_j=x_i$, contrary to our assumption
that they are distinct. The case $\alpha<0$ is handled similarly; in the final step
there, we get $x_j = -x_i$, which now contradicts the assumption that all the $x_i$ are positive.

We thus have an incidence problem for points and lines, with $k^2$ points and $O(k^2)$ distinct lines.
The Szemer\'edi-Trotter theorem (see (\ref{eq:st}) or Theorem \ref{thm:pachsharir}) gives
\[|I(S,\Gamma^+\cup\Gamma^-)| = O((k^2)^{2/3} (k^2)^{2/3} + k^2 + k^2) = O(k^{8/3}).\]
Taking into account the discarded quadruples, we have
\[|Q| = |I(S,\Gamma^+\cup\Gamma^-)| +O(k^2) = O(k^{8/3}),\]
so
\[ |f(W)| \ge \frac{k^4}{|Q|}=\upOmega(k^{4/3})=\upOmega(m^{4/3}),\]
completing the proof of Theorem \ref{thm:main2}.
\end{proof}



\newpage
\section{A bound on tangencies between circles}
\label{sec:circletangs}

In this section we study tangencies between circles.

\begingroup
\def\thetheorem{\ref{thm:circles1}}
\begin{theorem}
Given a family $C$ of $n$ circles in $\R^2$ with arbitrary radii,
there are at most $O(n^{3/2})$ points where at least two circles of $C$ are tangent.
\end{theorem}
\addtocounter{theorem}{-1}
\endgroup
\begin{proof}
Let the \emph{tangency graph} of $C$ be the 
graph whose vertex set is $C$,
with an edge between two circles if they are tangent.

First we modify the tangency graph by removing certain edges.
For any point where three or more circles are mutually tangent, we remove all but one (arbitrarily chosen) edge from the corresponding edges of the tangency graph, and we call the resulting subgraph $T$.
Then an upper bound on the number of edges of $T$ will still be an upper bound on the number of tangency points.
Moreover, the graph $T$ contains no $K_{3,9}$ or $K_{9,3}$.
Indeed, if three circles are not mutually tangent at a common point, then they have at most eight common
neighbors in $T$, by Lemma \ref{lem:apollonius}.
On the other hand, if three circles are mutually tangent
at some point $q$, then any common neighbor must be tangent to all three at $q$;
but then, by construction, at most one of the corresponding edges is in $T$.

For each circle $c \in C$, with center $(a,b)$ and radius $r$,
we define a point $p_c = (a,b,r)\in \Re^3$ and two surfaces in $\R^3$ given by
\begin{align*}
\sigma_c^+ &= \{ (x,y,z) \mid (x-a)^2+(y-b)^2 = (z+ r)^2\},\\
 \sigma_c^- &= \{ (x,y,z) \mid (x-a)^2+(y-b)^2 = (z- r)^2\}.
\end{align*}
These surfaces are cones.
Put $P = \{p_c \mid c\in C\}$ and $\Sigma = \{\sigma_c^+\mid c\in C\}\cup \{\sigma_c^-\mid c\in C\}$.

Two circles $c_1,c_2\in C$ are tangent if and only if the point $p_{c_1}$ is incident to one
of the surfaces $\sigma_{c_2}^+,\sigma_{c_2}^-$, and also $p_{c_2}$ is incident to $\sigma_{c_1}^+$ or $\sigma_{c_1}^-$.
Thus the number of edges in the incidence graph of $P\times \Sigma$ is twice the number of edges in the tangency graph of $C$.
Let $I$ be the subgraph of the incidence graph corresponding to the subgraph $T$, i.e., an incidence is in $I$ if the corresponding tangency is in $T$.
As noted, $T$ does not contain a copy of $K_{3,9}$ or $K_{9,3}$, which implies the same for $I$.
It follows from Theorem \ref{thm:zahl} that the number of tangency points is
\[|E(T)| = \frac{1}{2}|E(I)| = O\left(n^{3/4}n^{3/4}\right)=O\left(n^{3/2}\right).\]
This completes the proof.
\end{proof}

\begingroup
\def\thetheorem{\ref{thm:circles2}}
\begin{theorem}
Let $C$ be a family of $n$ circles in $\R^2$ with arbitrary radii.
Assume that no three circles of $C$ are mutually tangent at a common point.
Then $C$ has at most $O(n^{3/2})$ pairs of tangent circles.
\end{theorem}
\addtocounter{theorem}{-1}
\endgroup
\begin{proof}
We proceed exactly as in the proof of Theorem \ref{thm:circles1}, but now we do not need to remove any edges.
By Lemma \ref{lem:apollonius} and the condition on the circles, the tangency graph contains no $K_{3,9}$ or $K_{9,3}$.
Then Theorem \ref{thm:zahl} gives the upper bound $O(n^{3/2})$ for the number of tangent pairs.
\end{proof}

Let us observe that a stronger conclusion holds in Theorem \ref{thm:circles2}: The number of incident point-line pairs, such that at least two circles of $C$ are tangent at the point to the line (and thus to one
another), is $O(n^{3/2})$.


\section{Point-plane and point-line distances in \texorpdfstring{$\Re^3$}{space}}
\label{sec:highdim}
In this section we study two possible generalizations of the point-line distance problem to $\R^3$, to point-\emph{line} distances and point-\emph{plane} distances.
We treat point-plane distances first, as there our approach is more similar to that for point-line distances in $\R^2$.

\subsection{Distinct point-plane distances in \texorpdfstring{$\Re^3$}{space}}\label{subsec:pointplane}
\label{sec:pointplane}

One approach to point-plane distances would be to consider, for each plane $\pi$, the set $F_\pi$ of points at a fixed distance from $\pi$; we would then want to bound the number of incidences between the points and the sets $F_\pi$.
However, $F_\pi$ would be a pair of planes, each of which might be contained in many different sets $F_\pi$.
This would make it difficult to apply an incidence bound.
Instead, we consider tangencies of the planes with spheres around the points, and we dualize to obtain an incidence problem.

We first prove an upper bound on plane-sphere tangencies.
Given a finite set $\Pi$ of planes and a finite set $S$ of spheres, we write $T(\Pi,S)$ for
the number of pairs $(\pi,s) \in \Pi\times S$ such that $\pi$ is tangent to $s$, and we define the
\emph{tangency graph} to be the bipartite graph whose vertex classes are $\Pi$ and $S$, with an edge
between $\pi$ and $s$ if $\pi$ is tangent to $s$.

\begin{lemma} \label{lem:sphereplane}
Let $\Pi$ be a set of $n$ planes and let $S$ be a set of $k$ spheres in $\Re^3$.
Assume that the tangency graph of $\Pi\times S$ contains no $K_{M,3}$ or $K_{3,M}$,
for some constant parameter $M$. Then
\[T(\Pi,S) = O\left ( n^{3/4}k^{3/4} + n + k \right),\]
where the constant of proportionality depends on $M$.
\end{lemma}
\begin{proof}
The proof is similar to the proofs in Subsections \ref{subsec:firstbound} and \ref{subsec:linecirctangs}.
We apply a rotation so that all planes in $\Pi$ are non-vertical.
We map a plane $\pi \in \Pi$ defined by $z=ax+by+c$ to the point $\pi^* = (a,b,c)$.
We map each sphere $s \in S$ to the surface $s^*$ which is the locus of all points $\pi^*$ dual to planes $\pi$ tangent to $s$.
The surface $s^*$ dual to the sphere $s$ with center $(x,y,z)$ and radius $r$ is the set of points $(a,b,c)$ satisfying
\[
\frac{|z-ax-by-c|}{\sqrt{1+a^2+b^2}} = r ,\quad\quad\text{or}\quad\quad
(z-ax-by-c)^2 - r^2(1+a^2+b^2) = 0,
\]
which is a smooth quadric (specifically, a two-sheeted hyperboloid) in $\Re^3$.

Denote the set of resulting points by $\Pi^* = \{ \pi^* \mid \pi \in \Pi\}$, and the family of resulting
hyperboloids by $S^* = \{ s^* \mid s \in S\}$.
Every pair $(\pi,s)$ that is counted in $T(\Pi,S)$ corresponds to an incidence between the
point $\pi^*$ and the surface $s^*$, and the tangency graph of $\Pi\times S$ is the same as the incidence graph of $\Pi^*\times S^*$.
Thus applying Theorem \ref{thm:zahl} to $I(\Pi^*,S^*) = T(\Pi,S)$ gives the stated bound.
\end{proof}

Recall from Section \ref{sec:intro}
that we call a configuration of $s$ points and $t$ planes an $s\times t$  \emph{cone/cylinder configuration}
if the $s$ points lie on a line $\ell$, and the $t$ planes are all tangent to the same cone/cylinder with axis $\ell$.
In both types of configurations, there is a complete bipartite graph in the tangency graph between the planes and the spheres, centered at the points, that are tangent to the cone or cylinder.

\begin{corollary}\label{cor:planesspheres}
Let $\Pi$ be a set of $n$ planes and $S$ a set of $k$ spheres in $\Re^3$, that do not determine a $3\times 3$ cone or cylinder configuration.
Then
\[T(\Pi,S) = O\left ( n^{3/4}k^{3/4} + n + k \right).\]
\end{corollary}
\begin{proof}
By Lemma \ref{lem:threespheres}, there is no $K_{9,3}$ or $K_{3,9}$ in the tangency graph of $\Pi$ and $S$.
Thus Lemma \ref{lem:sphereplane} gives the stated bound.
\end{proof}

We now deduce a lower bound on point-plane distances, under the condition that the points and planes do not determine any cone or cylinder configurations.
As mentioned in Section \ref{sec:intro}, a cylinder configuration determines only one point-plane distance, so any non-trivial lower bound must have some condition that excludes cylinder configurations.
On the other hand, a cone configuration determines at least half as many distances as it has points on the axis of the cone,
so it does not seem necessary to have any restriction on cone configurations at all.
However, because our proof goes via Corollary \ref{cor:planesspheres}, we have no choice but to exclude them too.

\begingroup
\def\thetheorem{\ref{thm:3Dplanes}}
\begin{theorem}
Let $P$ be a set of $m$ points and let $\Pi$ be a set of $n$ planes in $\Re^3$.
Assume that $P$ and $\Pi$ do not determine a $3\times 3$ cone or cylinder configuration.
Then the number of distinct point-plane distances determined by $P$ and $\Pi$ is
\[\upOmega\left(\min\left\{m^{1/3}n^{1/3},n \right\}\right),\]
unless $m=O(1)$.
\end{theorem}
\addtocounter{theorem}{-1}
\endgroup
\begin{proof}
The proof is analogous to the proof in Subsection \ref{subsec:firstbound}. Let $t$ denote the total
number of distinct distances between points in $P$ and planes in $\Pi$. We place at most $t$
spheres centered at each $p\in P$ according to the occurring distances from $p$ to the planes in $\Pi$.
Let $S$ denote the resulting family of at most $mt$ distinct spheres.

On the one hand, the number of tangencies satisfies $T(\Pi,S) = mn$.
On the other hand, applying Corollary \ref{cor:planesspheres} to $\Pi$ and $S$ gives
\[mn = T(\Pi,S) = O\left(n^{3/4} (mt)^{3/4} +n+ mt\right).\]
The second term gives $m = O(1)$ and the third term gives $t = \upOmega (n)$.
Thus we have
\[ t = \upOmega\left(\min\left\{m^{1/3}n^{1/3},n\right\}\right),\]
unless $m=O(1)$.
\end{proof}

As in the case of point-line distances, our proof yields a stronger statement:
Under the same assumptions, there is a point $p \in P$ with at least $\upOmega(m^{1/3}n^{1/3})$ distinct distances from $p$ to the planes in $\Pi$.


\subsection{Distinct point-line distances in \texorpdfstring{$\Re^3$}{space}}

Finally, we consider distances between points and lines in $\R^3$.
There are two possible approaches we could take.
One approach would be to consider cylinders around the lines, so that we get an incidence problem between points in $\R^3$ and cylinders.
To apply an incidence bound, we would need a geometric fact of the form ``five points are all contained in at most a bounded number of cylinders, unless \ldots''; unfortunately, we do not know an appropriate fact of this form.
Note that a set of points on a line and a set of cylinders containing that line determine a complete bipartite incidence graph.

The second approach, which is the one we take, is to consider tangencies between spheres and lines.
The analogue of dualizing would be to move from lines to points in the parameter space of lines in $\R^3$, which is technically more complicated.
Instead, we move to a simple parameter space for spheres, by sending a sphere to the point in $\R^4$ consisting of the center of the sphere and its radius.
 This is analogous to the approach used for circles in Section \ref{sec:circletangs}.

\begin{lemma} \label{lem:sphereline}
Let $L$ be a set of $n$ lines and let $S$ be a set of $k$ spheres in $\Re^3$.
Assume that the tangency graph of $L\times S$ contains no $K_{4,M}$,
for some constant parameter $M$. Then
\[T(L,S) = O\left ( n^{4/5+\eps}k^{4/5} + n + k \right),\]
for any $\eps>0$, where the constant of proportionality depends on $M$ and $\eps$.
\end{lemma}
\begin{proof}
Map the sphere with center $(a,b,c)$ and radius $r$ to the point $(a,b,c,r^2)\in \R^4$.
Map a line $\ell$ in $\R^3$ to the set $V_\ell$ of all points corresponding to spheres that are tangent to the line.
If we represent $\ell$ by a point $u_\ell$ on it and a unit vector $v_\ell$ in its direction,
then an easy calculation shows that
$$
V_\ell = \{ (a,b,c,r^2) \mid r^2 = \left| ((a,b,c)-u_\ell)\times v_\ell\right|^2 \} ,
$$
which is the equation of a paraboloid in the four-dimensional $(a,b,c,r^2)$-space.
It is also clear that all these paraboloids are distinct.

Let $P$ be the set of points corresponding to the spheres of $S$, and let $V$ be the set of
paraboloids $V_\ell$ corresponding to the lines $\ell\in L$.
We apply Theorem \ref{thm:FPSSZ1} to $P$ and $V$.
A sphere and a line are tangent if and only if the corresponding point and paraboloid are incident,
so the tangency graph of $S\times L$ is identical to the incidence graph of $P\times V$.
By assumption, this graph contains no $K_{4,M}$, so Theorem \ref{thm:FPSSZ1} gives the stated bound.
\end{proof}

Note that the $\eps$ in Lemma \ref{lem:sphereline} (and in what follows) could be removed using the result of Basu and Sombra \cite{BS} mentioned in Subsection \ref{subsec:incs},
but this would require checking further conditions.
Since the bound is not expected to be tight anyway, we relied on the more convenient incidence bound of Fox et al. \cite{FPSSZ}.

\begin{corollary}\label{cor:linesspheres}
Let $L$ be a set of $n$ lines and $S$ a set of $k$ spheres in $\Re^3$, such that there is no $4\times 4$ cone/cylinder/hyperboloid configuration of four lines and four centers of the spheres.
Then
\[T(L,S) = O\left ( n^{4/5+\eps}k^{4/5} + n + k \right),\]
for any $\eps>0$, where the constant of proportionality depends on $\eps$.
\end{corollary}
\begin{proof}
This follows from Lemma \ref{lem:sphereline} and Lemma \ref{lem:fourspheres}.
\end{proof}

We now arrive at our lower bound for point-line distances in space, under the condition that the points and lines do not determine any cone/cylinder/hyperboloid configuration (defined in Section \ref{sec:intro}).

\begingroup
\def\thetheorem{\ref{thm:3Dlines}}
\begin{theorem}
Let $P$ be a set of $m$ points and $L$ a set of $n$ lines in $\Re^3$.
Assume that $P$ and $L$ do not determine any $4 \times 4$ cone/cylinder/hyperboloid configuration.
Then the number of distinct point-line distances determined by $P$ and $L$ is
\[\upOmega\left(\min\left\{m^{1/4}n^{1/4-\eps},n\right\}\right),\]
for any $\eps>0$ (where the constant of proportionality depends on $\eps$), unless $m=O(1)$.
\end{theorem}
\addtocounter{theorem}{-1}
\endgroup
\begin{proof}
Let $t$ denote the total number of distinct distances between points in $P$ and lines in $L$.
We place at most $t$ spheres centered at each $p\in P$ according to the occurring distances from $p$ to the lines in $L$.
Let $S$ denote the resulting family of at most $mt$ distinct spheres.
Applying Corollary \ref{cor:linesspheres} to $L$ and $S$ gives
\[mn = T(L,S) = O\left(n^{4/5+\eps} (mt)^{4/5} +n+ mt\right).\]
The second term gives $m = O(1)$ and the third term gives $t = \upOmega (n)$.
Thus we have
\[ t = \upOmega\left(\min\left\{m^{1/4}n^{1/4-\eps},n\right\}\right),\]
unless $m=O(1)$.
\end{proof}




\end{document}